\documentclass{article}
\usepackage{graphicx} 
\usepackage{amsmath}
\usepackage{amsthm}
\usepackage{amssymb}
\usepackage{enumerate} 
\usepackage{todonotes}
\usepackage{subcaption}
\usepackage{multirow}
\usepackage{url}
\usepackage{makecell}
\usepackage{nicefrac}

\newif\ifarxiv
\arxivtrue

\ifarxiv
\usepackage{charter}
\usepackage{arxiv}
\fi

\usepackage{natbib}
\usepackage[ruled]{algorithm2e}
\DontPrintSemicolon

\usepackage[capitalize,nameinlink,noabbrev]{cleveref}

\newtheorem{theorem}{Theorem}

\newtheorem{corollary}[theorem]{Corollary}
\newtheorem{proposition}[theorem]{Proposition}

\newcommand{\R}{\mathbb{R}}

\DeclareMathOperator*{\argmax}{arg\,max} 
\DeclareMathOperator*{\argmin}{arg\,min} 
\DeclareMathOperator*{\convexh}{conv}

\newcommand{\innp}[2]{\left\langle #1, #2 \right\rangle}

\newcommand{\conv}[1]{\convexh\left( #1\right)}
\newcommand{\eps}{\varepsilon}

\newcommand{\Xset}{{\ensuremath{\mathcal{X}}}\xspace}

\title{A Frank-Wolfe Algorithm for Oracle-based Robust Optimization}
\author{Mathieu Besan\c{c}on\\
Université Grenoble Alpes, Inria, CNRS, LIG\\
Grenoble, France\\
\texttt{mathieu.besancon@inria.fr}
\And
Jannis Kurtz \\
Amsterdam Business School, University of Amsterdam\\
Amsterdam, Netherlands\\
\texttt{j.kurtz@uva.nl}
}
\date{}

\begin{document}

\maketitle

\begin{abstract}
We tackle robust optimization problems under objective uncertainty in the oracle model, i.e.,~when the deterministic problem is solved by an oracle. The oracle-based setup is favorable in many situations, e.g., when a compact formulation of the feasible region is unknown or does not exist. We propose an iterative method based on a Frank-Wolfe type algorithm applied to a smoothed version of the piecewise linear objective function. Our approach bridges several previous efforts from the literature, attains the best known oracle complexity for the problem and performs better than state-of-the-art on high-dimensional problem instances, in particular for larger uncertainty sets.
\end{abstract}

\section{Introduction}
Optimization under uncertainty has been gaining a lot of attraction in the last decades, allowing for decision-making frameworks which can hedge against fluctuating or unknown parameters.
Robust optimization (RO) tackles situations in which the decision-maker can describe a set of possible values for the unknown parameters, but not necessarily their distribution \citep{ben2009robust,kouvelis2013robust,bertsimas2022robust,goerigk2024introduction}.
We consider objective-robust optimization problems of the form
\begin{equation}\label{eq:RO}\tag{RO}
    \min_{x\in \mathcal X} \max_{c\in \mathcal U} c^\top x,
\end{equation}
where $\mathcal X, \mathcal U\subset\R^n$ are compact convex sets. We call $\mathcal X$ the feasible region and $\mathcal U$ the uncertainty set.

Consider first the case where a description of the feasible region $\mathcal X$ is known, e.g., it can be described by a polynomial number of convex inequalities. In this case, for several convex uncertainty sets such as polyhedra and ellipsoids, the dualization technique yields a closed-form expression of the problem, i.e.,~dualizing the inner maximum expression, merging the corresponding minimization problem with the outer minimum, and obtaining a convex minimization problem which can be solved by off-the-shelf solvers.
While this reformulation can often be solved very efficiently, it requires a polynomial-sized description of the feasible region $\mathcal X$.
This is not necessarily the case in many relevant situations, in which the feasible set can be accessed only or more efficiently by a linear optimization oracle.
We provide three examples in robust optimization where the problem form we consider directly applies.
\paragraph*{Combinatorial Robust Optimization.}\label{ex:RO_lowerbound}
\;\newline
Combinatorial robust optimization problems are defined as
\begin{equation}\label{eq:comb_ro}
    \min_{x\in \mathcal Z} \max_{c\in \mathcal U} c^\top x
\end{equation}
where $\mathcal Z\subseteq \{ 0,1\}^n$ and $\mathcal U$ is a convex uncertainty set. This type of problem was intensively studied for a wide range of combinatorial problems; \cite{kouvelis2013robust,buchheim2018robust}.
The results indicate that robust combinatorial optimization problems are often NP-hard, even if the underlying combinatorial problem can be solved in polynomial time.
When using a branch-and-bound method to tackle such problems, tighter lower bounds can be obtained by replacing the continuous relaxation of Problem~\eqref{eq:comb_ro}
with:
\begin{equation}\label{eq:min-max_conv}
    \min_{x\in \conv{\mathcal Z}} \max_{c\in \mathcal U}\; c^\top x.
\end{equation}
This approach was applied e.g.~in \cite{bettiol2023oracle,al2020frank}.
Problem \eqref{eq:min-max_conv} is of the form \eqref{eq:RO}; however, since $\mathcal Z$ can represent any combinatorial problem, a polynomial-sized description of the set $\conv{\mathcal Z}$ may be unknown or does not exist, rendering the dualization approach intractable. On the other hand, linearly optimizing over $\conv{\mathcal Z}$ is equivalent to linearly optimizing over $\mathcal Z$, which can often be done by problem-specific algorithms.

\paragraph*{Min-max-min Robust Optimization.}\label{ex:RO_lowerbound}
\;\newline
The min-max-min robust optimization problem was introduced in \cite{buchheim2017min} and can be modeled as:
\[
\min_{x^1,\ldots ,x^k\in \mathcal Z}\max_{c\in \mathcal U}\min_{i=1,\ldots ,k} c^\top x^i
\]
where $k\in \mathbb N$ and $\mathcal Z\subseteq \{ 0,1\}^n$. The authors show that for $k\ge n+1$ the problem is equivalent to:
\begin{equation*}
\min_{x\in \conv{\mathcal Z}}\max_{c\in \mathcal U} \ c^\top x. 
\end{equation*}
which is of the form \eqref{eq:min-max_conv}.

\paragraph*{Two-stage Binary Robust Optimization.}\label{ex:two-stage_RO}
\;\newline
In two-stage binary robust optimization under objective uncertainty, two types of variables are considered, the here-and-now decisions $x$, which have to be determined before uncertainty is realized, and the wait-and-see decisions $y$ which can be determined after. These problems can be modeled as min-max-min problems of the form:
\begin{equation}\label{eq:two-stage_RO}
    \min_{x\in \mathcal Z} \max_{c\in \mathcal U} \min_{y\in \mathcal Y(x)} c^\top x + d^\top y,
\end{equation}
where $\mathcal Z\subseteq \{ 0,1\}^{n_x}$ is the feasible set for the here-and-now decisions, $\mathcal U$ is a convex uncertainty set and $\mathcal Y(x)\subseteq \{ 0,1\}^{n_y}$ is the feasible set of wait-and-see decisions, depending on the here-and-now decision $x$.
In \cite{kammerling2020oracle}, it was shown that Problem~\eqref{eq:two-stage_RO} can be solved by a branch \& bound method, where the branching is only performed over the here-and-now decisions $x$. In each of the nodes of the branch \& bound tree a lower bound is calculated by solving: 
\[
\min_{(x,y)\in \conv{\mathcal Z \times \mathcal Y(x)}} \max_{c\in \mathcal U} \ c^\top x + d^\top y
\]
which is again a problem of the type \eqref{eq:min-max_conv} where usually no polynomial description of the set $\conv{\mathcal{Z} \times \mathcal{Y}(x)}$ is known.

All the above examples motivate the use of so-called \textit{oracle-based algorithms}, where we assume that the feasible set $\mathcal X$ can only be accessed by a linear optimization oracle. More precisely, a linear optimization oracle is an algorithm which returns for every $c\in \mathcal{U}$ an optimal solution of the deterministic problem
$\min_{x\in\mathcal X} \ c^\top x$,
alleviating the need for a polynomial-sized formulation of $\mathcal X$.
Instead, any linear optimization algorithm over $\mathcal X$ can be used.
Hence, combinatorial algorithms for the underlying deterministic problem can be used.

For our analysis, we assume that each call to the oracle can be performed in constant runtime. We call an algorithm which solves \eqref{eq:RO} \textit{oracle-polynomial} if it has a polynomial runtime in the input parameters under the assumption that receiving an optimal solution of the deterministic problem by the oracle requires constant runtime.

Oracle-based algorithms have several properties of interest. As mentioned above, no mathematical optimization formulation of the feasible region is required.
If specialized algorithms were designed for the deterministic problem, they can be used directly in the oracle-based algorithm for the robust problem \eqref{eq:RO}.
Furthermore, an analysis of the runtime of the oracle-based algorithm gives insights into the connection between the complexity of the deterministic problem and the robust optimization problem. More precisely, if an oracle-polynomial algorithm exists for \eqref{eq:RO} then the problem can be solved in polynomial time for every deterministic problem which can be solved in polynomial time.

Our contributions are the following:
\begin{itemize}
    \item We design an oracle-based algorithm for \eqref{eq:RO} based on Frank-Wolfe applied to a smoothed version of the problem, unifying and connecting several previous lines of work.
    \item We derive a bound on the number of oracle calls needed by our approach that matches the current best known bound.
    \item We derive the first bound on the number of oracle calls needed to solve min-max-min robust optimization problems.
    \item We test our method on several instance sets showing that it outperforms other state-of-the-art methods on large-dimensional instances and in the high-uncertainty regime, i.e., when the set $\mathcal U$ is larger.
\end{itemize}

\subsection{Preliminaries}
\subsubsection{Notation and Definitions}
In the following we denote by $\|\cdot \|$ the Euclidean norm and define the following parameters: the diameter of the set \Xset is defined as
$D:=\max_{x,y\in \mathcal X} \| x-y\|$
and the maximum diameter of the uncertainty set is defined as 
$M:=\max_{c,c'\in\mathcal U} \| c-c'\|$. The maximum feasible solution length is defined as $D_{\max}:=\max_{x\in \mathcal X} \| x\|$. We denote by $\conv{\mathcal S}$ the convex hull of a set $\mathcal S$.

\subsubsection{Frank-Wolfe Algorithms}
Frank-Wolfe (FW) algorithms optimize nonlinear differentiable functions over a compact convex set
and have gained significant traction in optimization and machine learning in the last decade \citep{jaggi2013revisiting,bomze2021frank,braun2022conditional}.
This success can in part be explained by the flexible assumptions the algorithm requires on the problem representation.
Namely, only first-order information on the objective function and access to a \emph{Linear Minimization Oracle} (LMO) are required which, given a linear objective, computes an extreme point of the feasible region minimizing this objective.

Optimizing a convex, $L$-Lipschitz-smooth function $f$ over a compact convex set
of diameter $D$ results in a primal optimality gap after $t$ iterations of
\begin{align*}
    f(x_t) - f^* \leq \frac{2LD^2}{t+2}
\end{align*}
when using the agnostic step size with the standard FW algorithm \citep[Remark 2.3]{braun2022conditional}.\\

Frank-Wolfe methods are known to fail to converge when the objective function is only subdifferentiable, with an example proposed in \citet{nesterov2018complexity}, which, importantly for our problem, can be seen as an instance of (RO) with budgeted uncertainty.
\citet{yurtsever2018conditional} established a framework for the minimization of the sum of a differentiable function and a convex function with a well-defined proximal operator with constraints under a linear oracle model based on Frank-Wolfe. The authors show a convergence in $\mathcal{O}(\varepsilon^{-2})$ iterations based on an adaptive smoothing parameter. The difference in convergence rate is negligible compared to a fixed smoothing parameter based on the targeted accuracy.

\subsubsection{Oracle-based Robust Optimization}
Several oracle-based algorithms were derived for robust optimization problems of the form \eqref{eq:RO} for convex or discrete sets $\mathcal X$. In \citet{bertsimas2003robust} the authors show that for $\mathcal X\subseteq \{ 0,1\}^n$ and for a budgeted uncertainty set $\mathcal U$, Problem~\eqref{eq:RO} can be solved with $n+1$ oracle calls. This result was later improved in \citet{lee2014short} to at most $\lceil\frac{n-\Gamma}{2}\rceil +1 $ oracle calls, where $\Gamma$ is the number of uncertain parameters which can deviate from its mean value at the same time. The result was generalized in \cite{poss2018robust} to uncertainty sets defined by a fixed number of knapsack constraints.

To solve \eqref{eq:RO} with convex $\mathcal X$, \cite{buchheim2017min} introduces a constraint generation algorithm (CGA) which is performed on the dual problem of \eqref{eq:RO}. Deriving a new cut at each iteration can be done by minimizing the deterministic problem, i.e., calling the LMO. This algorithm turns out to perform very well on moderate sized combinatorial problems or when the uncertainy budget $\Gamma$ is small.

\cite{ben2015oracle} developed an oracle-based algorithm for robust optimization problems under constraint uncertainty, using tools from online convex optimization.
When applied to \eqref{eq:RO}, the presented method is equivalent to a projected subgradient method performed on the dual problem of \eqref{eq:RO}, in the uncertainty set space.
It finds an $\epsilon$-optimal point in the dual in at most $\nicefrac{D_{\max}^2M^2}{\eps^2}$ iterations. For constraint uncertainty, an additional binary search has to be performed over the optimal value, increasing the bound on the number of iterations by a factor of $\log \left(\frac{1}{\eps}\right)$. At each iteration, the algorithm calls a linear optimization oracle over $\mathcal X$ and a projection oracle over $\mathcal U$.

In \cite{buchheim2018frank}, a Frank-Wolfe type method is used to solve mean-risk optimization problems, which are equivalent to \eqref{eq:RO} with $\mathcal U$ being an ellipsoid and $\mathcal X$ convex and defined by one knapsack constraint. The algorithm uses a minimization oracle over $\mathcal X$.

In \cite{al2020frank}, Frank-Wolfe is applied to Problem \eqref{eq:RO} under ellipsoidal uncertainty. Since for ellipsoidal uncertainty, and more generally strongly-convex sets, the objective function of \eqref{eq:RO} is differentiable, the classical FW framework can be applied.

Finally, in \cite{bettiol2023oracle} a simplicial decomposition algorithm is presented, which can be compared to a sub-gradient version of the Fully-Corrective Frank-Wolfe algorithm.
Iteratively, the algorithm uses a linear minimization oracle over $\mathcal X$ and a convex hull oracle which returns an optimal solution to the problem
\begin{equation}\label{eq:min-max_bettiol}
\min\limits_{x\in\conv{x^1,\ldots ,x^k}} \max_{c\in \mathcal U} c^\top x.
\end{equation}
Since the latter problem is the dual of Problem \eqref{eq:dual_ro_epigraph}, the algorithm of \cite{bettiol2023oracle} is performing the same major steps as the constraint generation algorithm in \cite{buchheim2017min}; see Section \ref{sec:equivalence_algorithms} for a more detailed discussion. The authors additionally incorporate a technique to remove solutions from the previous iterates from Problem \eqref{eq:min-max_bettiol}.

A summary of the aforementioned methods is presented in Table~\ref{tab:overview_oracle_algorithms} which shows the class of feasible region $\mathcal{X}$, the class of uncertainty sets $\mathcal{U}$, the required oracles, and the maximum number of oracle calls needed to ensure optimality (up to an additive accuracy of $\eps>0$).

\begin{table*}[h!]
    \centering
    \resizebox{1.1\columnwidth}{!}{
    \begin{tabular}{cc|ccccc}
        Method & References & $\mathcal X$ & $\mathcal U$ & Oracle $1$ & Oracle $2$ & \# Oracle Calls \\
        \hline 
        \makecell{List of deterministic\\ problems} & \makecell{\cite{bertsimas2003robust}, \\ 
        \cite{lee2014short}, \\ \cite{alvarez2013note}}  & $\subseteq \{ 0,1\}^n$ & budgeted & $\min\limits_{x\in\mathcal X} c^\top x$ & - & $\lceil\frac{n-\Gamma}{2}\rceil +1 $\\
        \hline
        \makecell{List of deterministic \\ problems} & \cite{poss2018robust} & $\subseteq \{ 0,1\}^n$ & \makecell{$s$ knapsack \\ constraints} & $\min\limits_{x\in\mathcal X} c^\top x$ & - & $\mathcal O\left(s^sn^s\right)$\\
        \hline
        \makecell{Projected \\ subgradient descent} & \cite{ben2015oracle} & convex  & convex & $\min\limits_{x\in\mathcal X} c^\top x$ & $\min\limits_{c\in\mathcal U} \|c-\hat c\|^2$ & $\frac{D_{\text{max}}^2 M^2}{\eps^2}$\\
        \hline
        \makecell{Away Frank-Wolfe}  & \cite{buchheim2018frank} & $\left\{ y\ge 0: \ a^\top y \le b\right\}$ & ellipsoid & $\min\limits_{x\in\mathcal X} c^\top x$ & $\max\limits_{x\in\mathcal X} c^\top x$ & - \\
        \hline
        \makecell{Vanilla Frank-Wolfe}  & \cite{al2020frank} & convex & ellipsoid & $\min\limits_{x\in\mathcal X} c^\top x$ & - & - \\
        \hline
        \makecell{Simplicial decomposition based\\ Frank-Wolfe type}  & \cite{bettiol2023oracle} & convex & convex & $\min\limits_{x\in\mathcal X} c^\top x$ & $\min\limits_{x\in\conv{x^1,\ldots ,x^k}} \max_{c\in \mathcal U} c^\top x$ & - \\
        \hline
        \makecell{Smoothing \\ Frank-Wolfe}  & this work & convex & convex & $\min\limits_{x\in\mathcal X} c^\top x$ & $\min\limits_{c\in\mathcal U} \|c-\hat c\|^2$ & $\frac{4D^2M^2}{\eps^2}$
    \end{tabular}}
    \caption{Overview of oracle-based algorithms for robust optimization problems of the form \eqref{eq:RO}.}
    \label{tab:overview_oracle_algorithms}
\end{table*}

Oracle-based algorithms were also used to solve two-stage robust binary optimization problems \citep{kammerling2020oracle}, min-max-min robust combinatorial optimization problems \citep{buchheim2017min}, and for robust planning of production routing \citep{borumand2024oracle}. In \cite{buchheim2020note}, the author proves that in general no oracle-polynomial algorithm exists for Problem \eqref{eq:RO} if the uncertainty set contains a finite number of scenarios and $\mathcal X$ is binary. 

\subsection{Equivalence between Algorithms}\label{sec:equivalence_algorithms}
In the following, we present a more detailed description of the constraint generation algorithm (CGA) developed in \cite{buchheim2017min} and the simplicial decomposition algorithm (SDA) developed in \cite{bettiol2023oracle} and show that CGA is performing the same steps as SDA but applied to the dual problem of \eqref{eq:RO}.

To solve \eqref{eq:RO} with convex $\mathcal X$, \citet{buchheim2017min} consider the dual problem of \eqref{eq:RO} given as
\begin{equation}\label{[eq:dual_RO}\tag{Dual-RO}
\max_{c\in\mathcal U}\min_{x\in \mathcal X} c^\top x.
\end{equation}
For a finite set of solutions $\mathcal X'\subset \mathcal X$, the latter problem can be solved via an epigraph reformulation as
\begin{equation}\label{eq:dual_ro_epigraph}
\begin{aligned}
    \max_{c,\tau} \ &\tau \\
    s.t. \quad &  \tau \leq c^\top x \quad \forall \ x\in\mathcal X' \\
    & c\in \mathcal U
\end{aligned}
\end{equation}
before iteratively adding new solutions from $\mathcal X$ which cut off the current optimal solution $(c^*,z^*)$.
Finding such a violating cut can be done by minimizing the deterministic problem with objective $c^*$, i.e., calling the LMO. This algorithm turns out to perform very well on moderate sized problems.

In contrast to CGA, SDA solves the primal version
\begin{equation}\label{eq:primal_RO}
\min_{x\in \conv{\mathcal X'}}\max_{c\in\mathcal U} c^\top x
\end{equation}
for the current finite subset $\mathcal X'\subset \mathcal X$ (instead of solving \eqref{eq:dual_ro_epigraph}), which leads to the same objective value as \eqref{eq:dual_ro_epigraph}. However, an optimal solution $x^*=\sum_{x\in\mathcal X'}\alpha_x x \in \conv{\mathcal X'}$ of \eqref{eq:primal_RO} and an optimal solution in $\argmax_{c\in \mathcal U}c^\top x^*$ has to be calculated in each iteration, while CGA only works with an optimal solution $c^*\in\mathcal U$. Afterwards, both algorithm apply the LMO to compute the next vertex $x_t\in \mathcal X$.
CGA only has to solve Problem~\eqref{eq:primal_RO} in the last iteration to obtain the final optimal solution.

Finally, we highlight that SDA itself is equivalent to fully-corrective Frank-Wolfe (FCFW) \citep{jaggi2013revisiting} applied to the original non-smooth function, in which the correction step (optimization over the active set) solves the auxiliary LP. However, note that the proof of convergence in finite time from \citet{bettiol2023oracle} is not based on the connection to FCFW. Furthermore, the proof of convergence of FCFW cannot be used to establish a convergence in oracle complexity of SDA, since it relies on the objective smoothness.

These two connections together do not come as a surprise, since duality between FCFW and the cutting plane algorithm was established in different contexts, e.g.~in \citet{zhou2018limited}.

\section{Oracle Algorithms for Robust Optimization via Smoothing}

In this section we first introduce a smoothed version of Problem~\eqref{eq:RO} using the techniques established in \cite{nesterov2005smooth}, which can be solved by classical variants of the Frank-Wolfe algorithm.
Afterwards, we derive the number of iterations (and hence oracle calls) we need to solve Problem~\eqref{eq:RO} up to an additive error of $\eps$.

\subsection{Objective Smoothing}

In the following we consider the smoothed robust optimization problem:
\begin{equation}\label{eq:smoothed_RO}
\min_{x\in \mathcal X}\max_{c\in \mathcal U} \ c^\top x - \frac{\mu}{2}\| c-c_0\|^2
\end{equation}
where $c_0\in \mathcal{U}$ is a fixed scenario and $\mu>0$ is a fixed smoothing parameter. We denote the objective function of the original problem \eqref{eq:RO} as 
\[
f(x):=\max_{c\in \mathcal U} \ c^\top x
\]
and of the smoothed robust problems as
\[
f_{\mu}(x):= \max_{c\in \mathcal U} \ c^\top x - \frac{\mu}{2} \| c-c_0\|^2.
\]
In contrary to $f$, the smoothed function $f_{\mu}$ is differentiable and Lipschitz-smooth by strong convexity of its conjugate;
see \cref{fig:smoothing} for an illustration.
Furthermore, its gradient is given as the unique optimal solution of the maximization problem over $\mathcal U$
\begin{equation}\label{eq:gradient_smoothed_problem}
   \nabla  f_{\mu}(x) = \argmax_{c\in \mathcal U} \ c^\top x - \frac{\mu}{2} \| c-c_0\|^2 . 
\end{equation}

\begin{figure}
    \centering
    \includegraphics[trim={0.7cm 0.5cm 1.0cm 1.5cm},clip, width=0.23\textwidth]{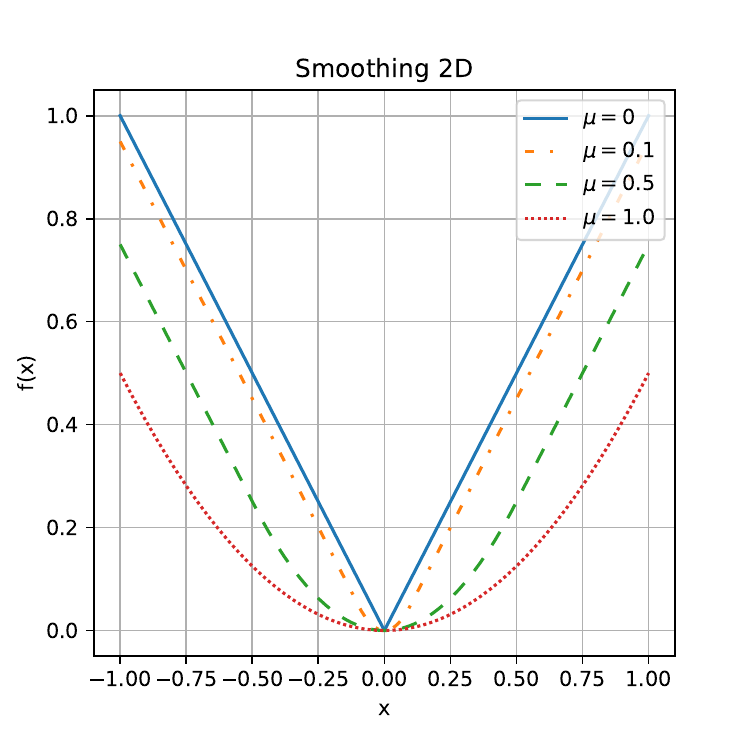}
    \caption{Smoothing of the function $f(x)=\max\limits_{-1\le c\le 1} c^\top x$ for different $\mu$ values.}
    \label{fig:smoothing}
\end{figure}

From \citet[Theorem 1]{nesterov2005smooth}, it follows that the gradient of $f_{\mu}$ is Lipschitz continuous with Lipschitz constant $L:=\frac{1}{\mu}$. The smoothed Frank-Wolfe algorithm is presented in \cref{alg:smoothed}.

\begin{algorithm}
\KwData{Linear oracle for $\mathcal{X}$, projection oracle for $\mathcal{U}$, $\mathbf{x}_0 \in\mathcal{X}$, $T > 0$, $\mu > 0$}
\For{$t \in 1\dots T$}{
    $g_t \gets \nabla f_{\mu}(\mathbf{x}_t)$ \hfill (calling the projection oracle on $\mathcal U$)\;
    $\mathbf{v}_t = \argmin\limits_{\mathbf{v}\in\mathcal{X}} \langle \mathbf{v}, g_t \rangle $ \hfill (calling the linear optimization oracle over $\mathcal X$)\;
    $\gamma_t = \frac{2}{t+1}$\;
    $\mathbf{x}_{t+1} = \mathbf{x}_{t} + \gamma_t (\mathbf{v}_{t} - \mathbf{x}_{t})$\;
}
\Return $\mathbf{x}_T$
\caption{Smoothed Frank-Wolfe}
\label{alg:smoothed}
\end{algorithm}

\subsection{Oracle Complexity via Frank-Wolfe Algorithms}
We now present the oracle complexity of the smoothed FW approaches to (RO).
\begin{theorem}\label{thm:number_iterations}
Let $\eps>0$ and $\mu=\frac{\eps}{M^2}$. Let $x_t\in \mathcal X$ be the solution calculated at the $t$-th iteration of the Frank-Wolfe algorithm applied to \eqref{eq:smoothed_RO} and let $x^*$ be an optimal solution of Problem~\eqref{eq:RO}. Then, in order to obtain a primal gap
\[
f(x_t)-f(x^*) \le \eps,
\]
we need $t$ iterations with
\[
t\ge \frac{4 D^2M^2}{\eps^2}.
\]
\end{theorem}
\begin{proof}
From the classical convergence analysis of the Frank-Wolfe algorithm, and since the gradient of $f_{\mu}$ is Lipschitz continuous with Lipschitz constant $\frac{1}{\mu}$, it follows that
\begin{equation}\label{eq:convergence_bound}
f_{\mu}(x_t) - f_{\mu}(x_{\mu}^*) \le \frac{2D^2\frac{1}{\mu}}{t+3},
\end{equation}
where $x_{\mu}^*$ is the optimal solution of \eqref{eq:smoothed_RO}; see e.g.~\cite{braun2022conditional}.
We have
\begin{align*}
    f(x_t)-f(x^*) &\le f(x_t) - f_{\mu}(x^*) \\
    & \le f(x_t) - f_{\mu}(x_{\mu}^*) \\
    & = f(x_t) - f_{\mu}(x_t)+ f_{\mu}(x_t) - f_{\mu}(x_{\mu}^*)
\end{align*}
where the first inequality holds since $f(x)\ge f_{\mu}(x)$ for all $x\in\mathcal X$ and the second inequality holds since $x_{\mu}^*$ is the minimizer of $f_{\mu}$ over $\mathcal X$. Let $c_t\in \mathcal{U}$ be an optimal solution of $\max_{c\in U} c^\top x_t$. We can combine both inequalities as
\begin{align*}
    & f(x_t) - f_{\mu}(x_t)+ f_{\mu}(x_t) - f_{\mu}(x_{\mu}^*) \\
    & \le c_t^\top x_t - c_t^\top x_t + \frac{\mu}{2}\|c_t-c_0\|^2 + f_{\mu}(x_t) - f_{\mu}(x_{\mu}^*)\\
    & \le \frac{\eps}{2} + f_{\mu}(x_t) - f_{\mu}(x_{\mu}^*) \\
    & \le \frac{\eps}{2} + \frac{2D^2\frac{1}{\mu}}{t+3} \\
    & \le \frac{\eps}{2} +\frac{\eps}{2} =\eps
\end{align*}
where the first inequality is a consequence of $c_t$ maximizing $c^\top x_t$ over $\mathcal U$, the second inequality follows from the definition of $\mu$ and since $\|c_t-c_0\|\le M$, the third inequality follows from \eqref{eq:convergence_bound} and the last inequality stems from
\[
\frac{2D^2\frac{1}{\mu}}{t+3}\le \frac{2D^2\frac{1}{\mu}}{t}
\]
and by substituting $\mu=\frac{\eps}{M^2}$ and $t\ge \frac{4 D^2M^2}{\eps^2}$.
\end{proof}

The following corollaries follow directly from \cref{thm:number_iterations}.
\begin{corollary}
If the optimization problem in \eqref{eq:gradient_smoothed_problem} can be solved in polynomial time, then the robust optimization problem \eqref{eq:RO} can be solved up to an accuracy of $\eps>0$ in oracle-polynomial time with at most $\frac{4 D^2M^2}{\eps^2}$ oracle calls.
\end{corollary}

\begin{corollary}
If the optimization problem in \eqref{eq:gradient_smoothed_problem} can be solved in polynomial time, then the min-max-min robust optimization problem with $k\ge n+1$ solutions can be solved in oracle polynomial time with at most $\frac{4 nM^2}{\eps^2}$ oracle calls.
\end{corollary}

The last corollary provides the first bound on the number of oracle calls which is needed to solve min-max-min robust combinatorial optimization problems. While the authors in \citet{buchheim2017min} prove that the number of oracle calls is polynomial, no useful bound was ever derived before.

Important to note is that the proof of Theorem \ref{thm:number_iterations} relies on the convergence rate of FW but not on the specific sequence of iterates produced by standard FW. Importantly, this implies that our approach is not tied to standard FW but can use any modern FW variant that is known to enjoy faster convergence rates in many settings.
We explain the smoothing algorithms with standard FW for simplicity but will use the blended pairwise conditional gradient (BPCG) \citep{tsuji2022pairwise} in the experiments.
One key issue of the scheme developed in \cref{thm:number_iterations} is the Lipschitz smoothness parameter growing proportionally with the inverse of $\varepsilon$, which hinders convergence.
A natural modification is to use an adaptive smoothness parameter $\mu_t$ decreasing with each iteration, starting from a smooth approximation of the function and decreasing to ensure convergence to the true optimum.
This idea is summarized in \cref{alg:adaptivesmooth}.
One remaining difficulty is maintaining convergence to an $\varepsilon$-solution by balancing the trade-off between the smoothness of the current function $f_\mu$ and the approximation of the optimal solution when decreasing $\mu$. \cref{thm:dynamic_smoothness} ensures the standard convergence of $\mathcal{O}(1/t)$ when using a smoothness decreasing as $(1+t)^{-\frac12}$.
\begin{algorithm}
\KwData{Linear oracle for $\mathcal{X}$, projection oracle for $\mathcal{U}$, \ $\mathbf{x}_0 \in\mathcal{X}$, $T > 0$, $D$, $L_{\mathcal{U}}$}
\For{$t \in 1\dots T$}{
    $\mu_t = \frac{2D}{L_{\mathcal{U}} \sqrt{t+1}}$\;
    $g_t \gets \nabla f_{\mu}(\mathbf{x}_t)$ \hfill (calling the projection oracle on $\mathcal U$)\;
    $\mathbf{v}_t = \argmin\limits_{\mathbf{v}\in\mathcal{X}} \langle \mathbf{v}, g_t \rangle $ \hfill (calling the linear optimization oracle over $\mathcal X$)\;
    $\gamma_t = \frac{2}{t+1}$\;
    $\mathbf{x}_{t+1} = \mathbf{x}_{t} + \gamma_t (\mathbf{v}_{t} - \mathbf{x}_{t})$\;
    \KwSty{end for}
}
\caption{Adaptive-smoothing robust Frank-Wolfe}
\label{alg:adaptivesmooth}
\end{algorithm}

\begin{proposition}\label{thm:dynamic_smoothness}

Define $M_{\max} := \max_{c\in\mathcal{U}} \|c\|$.
The solution $x_T \in \mathcal{X}$ obtained at the $T$-th iteration of \cref{alg:adaptivesmooth} applied to Problem \eqref{eq:RO} with the following smoothness schedule
\begin{align*}
    \mu_t = \frac{2D}{M_{\max}}\frac{1}{\sqrt{t+1}},
\end{align*}
results in a primal gap of at most:
\begin{align*}
f(x_T) - f^* \leq \frac{D M_{\max}}{2\sqrt{T}}.
\end{align*}
\end{proposition}
\begin{proof}
The proof follows directly from \citet[Theorem 3.2]{yurtsever2018conditional}, which considers a composite objective $\phi(x) + g(Ax)$ with $\phi$ smooth and $g$ Lipschitz-continuous. Our problem setting applies with $\phi(\cdot) = 0$, $A = I$, and $g(x) = \max\limits_{c\in\mathcal{U}} c^\top x$ which is Lipschitz-continuous with constant $M_{\max}$.
\end{proof}

One interesting aspect we want to highlight is that the FW-based convergence analysis can be generalized to the inexact LMO setting in which the vertices produced incur an additive or multiplicative error on the LMO subproblem.
This is particularly important for expensive oracles for which one can ease the computational burden by requiring approximate solutions.
Indeed, our analysis relies on a primal gap being reached by FW on the smoothed problem which can be ensured with a dependence on the subproblem error.
For the fixed smoothing algorithm, the inexact oracle settings with additive and multiplicative errors were analyzed in \citet{jaggi2013revisiting}, ensuring an $\varepsilon$-optimal solution to the original problem.
\citet{yurtsever2018conditional} also propose generalizations of \citet[Theorem 3.2]{yurtsever2018conditional} to the two types of oracle error.

\subsection{Function and Gradient Evaluations}

Evaluating $f_\mu$ and $\nabla f_\mu$ at $x$ requires solving the regularized adversarial problem:
\begin{align*}
    \max_{c \in \mathcal{U}} \, c^\top x - \frac{\mu}{2} \|c - c_0\|^2 \,\,\, \Leftrightarrow 
    \;\; \max_{c \in \mathcal{U}} \, - \frac{\mu}{2} \|c - (c_0 + \frac{x}{\mu})\|^2.
\end{align*}
The inner subproblem thus amounts to an Euclidean projection onto the uncertainty set.
Hence, for classical budgeted uncertainty of the form
\begin{align*}
    \mathcal{U}(d,\Gamma, \underline{c}) = \{ c \in \left[\underline{c},\underline{c} + d\right], \sum_j \frac{c_j - \underline{c}_j}{d_j} \leq \Gamma \}.
\end{align*}
the projection problem results in a quadratic knapsack problem. Dualizing the knapsack constraint makes the projection problem fully separable; this property is leveraged in particular by \emph{breakpoint algorithms} for diagonal quadratic convex continuous knapsack problems. We refer the reader to \citet{patriksson2015algorithms} for a recent review of solution approaches, from which we implement the two-pegging breakpoint algorithm.

In case the uncertainty set is the convex hull of a finite scenario set $\{c_s\}_{s\in S}$, projecting on $\mathcal U$ amounts to solving the quadratic problem:
\begin{align*}
    \min_{c, \lambda} \,\, & \frac{\mu}{2} \|c - (c_0 +  x / \mu)\|^2\\
    \text{s.t. } \,\,& \sum_{s\in S} \lambda_s c_s = c \\
    & \lambda \geq 0, \sum_{s\in S} \lambda_s = 1.
\end{align*}

\subsection{Faster Solutions and Dual Bounds from Convex Hulls}\label{sec:convhull}

Although the objective function is transformed from piecewise linear to smooth, its original structure can be exploited to accelerate the algorithm.
In particular, the FW-based method we propose provides a convergence rate in terms of gradient and linear oracle calls but does not ensure a convergence in finite time to the optimal solution. We augment the run of the algorithm with the computation of the minimizer of the original robust objective over the convex hull of the vertices observed throughout all FW iterations, similar to the oracle used in \citet{bettiol2023oracle}:
\begin{align*}
x_{\mathrm{conv}} \in \argmin\limits_{x\in\conv{v_1,\ldots ,v_t}} \max_{c\in \mathcal U} c^\top x,
\end{align*}
where $v_1,\ldots ,v_t \in \mathcal X$ are vertices obtained at previous iterations. The result of this subproblem can be used in two aspects.
First, the minimizer $x_{\mathrm{conv}}$ can provide an improved primal bound over the best FW solution found so far.
Second and most importantly, we can derive at $x_{\mathrm{conv}}$ a subgradient $c_{\mathrm{conv}}$ of $f$ and $v_{\mathrm{conv}} \in \argmin_{v \in \Xset} \innp{c_{\mathrm{conv}}}{v}$, producing a suboptimality gap $\innp{c_{\mathrm{conv}}}{x_{\mathrm{conv}} - v_{\mathrm{conv}}}$ which reaches 0 at the optimum (see \citet[Lemma 1]{bettiol2023oracle}). This second aspect plays a more decisive role in the computational experiments.
\\

We highlight that the suboptimality gap is akin in expression to the FW gap in the smooth case, allowing for a parallel to be drawn between the algorithm from \citet{bettiol2023oracle} and the Fully Corrective Frank-Wolfe algorithm which also optimizes over the convex hull of the current set of vertices at each iteration. A distinctive characteristic of the Simplicial Decomposition algorithm is that the convex hull problem is tractable exactly as a linear program and does not require handling error as the FCFW corrective step.

\section{Experiments}

We showcase the computational effectiveness of our algorithm on the robust minimum-weight spanning tree and traveling salesperson problems.
We compare the true function value $f(x)$ as iterations progress and focus on the constraint generation algorithm (\texttt{consgen}) from \citet{buchheim2017min}, a FW variant with fixed smoothing (\texttt{FW}), the same algorithm with adaptive smoothing (\texttt{A-FW}), and the fixed-smoothing FW algorithm with an additional convex hull solve and subgradient bound introduced in \cref{sec:convhull} (\texttt{FW-convhull}). We do not include the subgradient method from \citet{ben2015oracle} since it iteratively tightens a dual bound but does not produce a sequence of primal solutions.\footnote{The code for all experiments is made publicly available at \url{https://github.com/matbesancon/oracle_robust_smoothing_fw}.}

For the spanning tree example, the LMO uses the Kruskal algorithm as implemented in \texttt{Graphs.jl}.
We use the \texttt{FrankWolfe.jl}~\citep{besanccon2022frankwolfe} implementation of the blended pairwise conditional gradient with lazification, meaning the algorithm only calls the LMO when it cannot perform sufficient progress over the current set of vertices.

\begin{figure*}[t!]
    \centering
    \begin{subfigure}[t]{0.32\textwidth}
        \centering
        \includegraphics[height=1.1in]{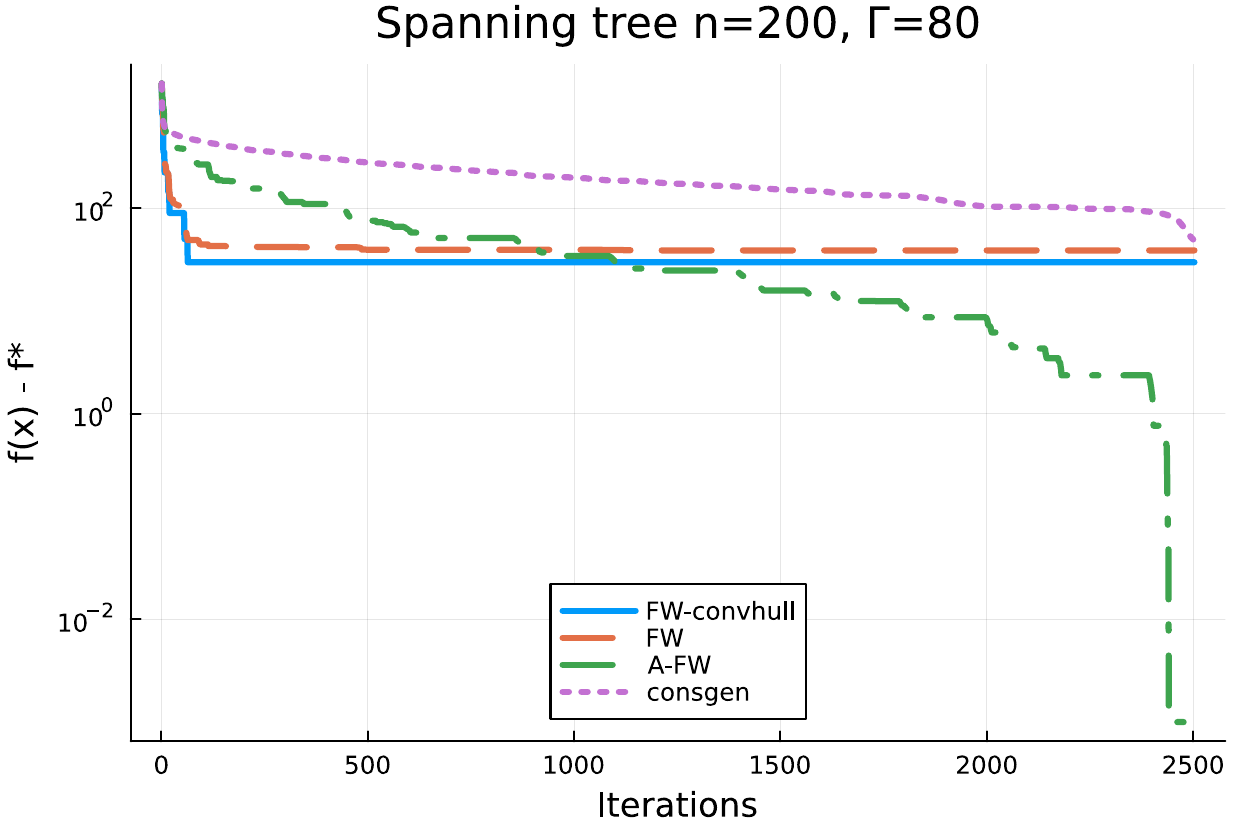}
        \caption{Iterations}
    \end{subfigure}
    \begin{subfigure}[t]{0.32\textwidth}
        \centering
        \includegraphics[height=1.1in]{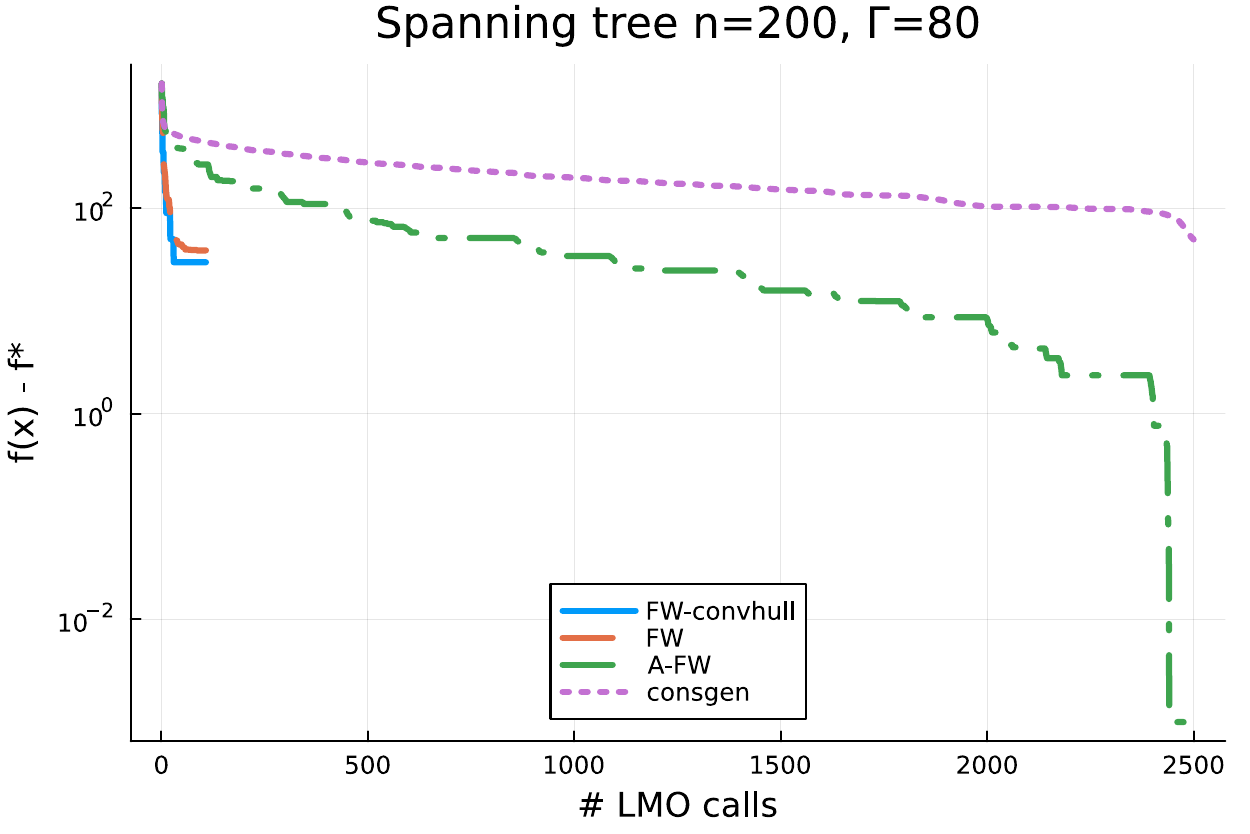}
        \caption{Oracle calls}
    \end{subfigure}
    \begin{subfigure}[t]{0.32\textwidth}
        \centering
        \includegraphics[height=1.1in]{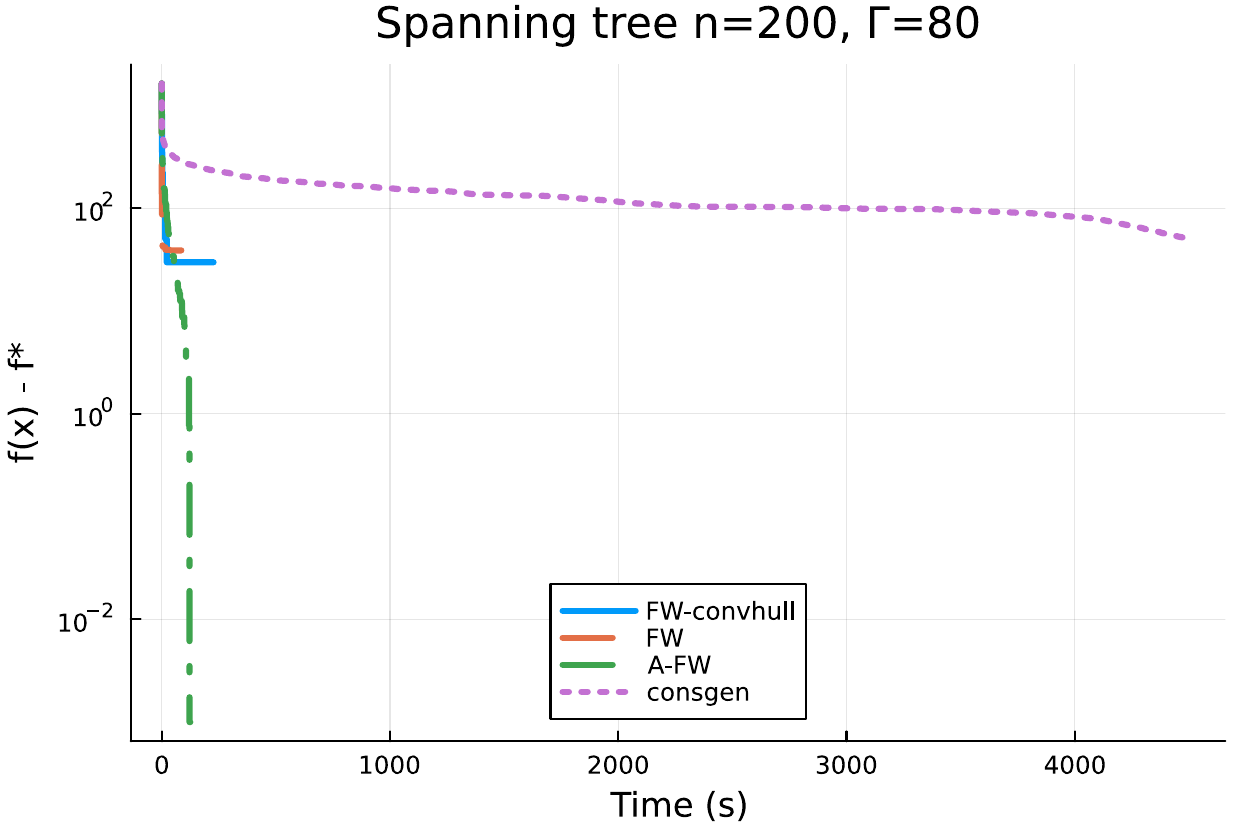}
        \caption{Runtime}
    \end{subfigure}
    \caption{Comparison of the constraint generation and FW-based algorithms on a robust spanning tree example. 
    }
    \label{fig:spanresults}
\end{figure*}

\cref{fig:spanresults} presents a comparison of the algorithms on the robust minimum-weight spanning tree problem in terms of iterations, runtime, and linear oracle calls.
All algorithms stop after a limit of 10000 iterations or 2500 LMO calls.
The limit on LMO calls is always the limiting one for \texttt{A-FW} and \texttt{consgen} since they both require one LMO call per iteration, also implying that these two algorithms potentially struggle more at larger scales for costly LMOs. On the other hand, FW algorithms with a fixed smoothing perform only about a hundred LMO calls throughout their iterations, which is typical for lazified FW variants \citep{braun2019lazifying}.
The constraint generation algorithm incurs a high cost per iteration which we do not observe with \texttt{A-FW} at an equivalent number of LMO calls. The explanation lies in the cost of the linear problem \eqref{eq:dual_ro_epigraph} solved at every iteration for which the number of constraints grows with the number of iterations. Hence, this LP can become intractable at larger scale and stall all progress.
Many instances even stop due to the memory requirements. On the contrary, the cost per iteration of FW methods is vastly dominated by the LMO since the gradient computation amounts to the projection onto the uncertainty set which remains tractable for many typical sets of interest, including the budgeted uncertainty set.
For smaller instances, the constraint generation reaches the exact optimum, benefiting from its finite time property unlike the smoothing approach combined with iterative algorithms.

We also study the influence of $\Gamma$ on the performance of the different methods for a fixed dimension, illustrated in \cref{fig:spanresultsgamma}, on 2500 iterations only.
We observe a decreased relative performance of constraint generation when the uncertainty budget $\Gamma$ increases, while it converges very fast for low $\Gamma$ values.
This performance behavior can similarly be tied to the LP solved in the space of uncertain parameters: since the uncertainty set is larger, more constraints corresponding to vertices of $\mathcal X$ have to be generated and added. This behavior is aligned with the observations made in \citet{buchheim2017min}.
We also note a large difference in performance between \texttt{FW} and \texttt{FW-convhull} for $\Gamma=60$, the subproblem over the convex hull effectively computing better primal solutions for that instance.

\begin{figure*}[t!]
    \centering
    \begin{subfigure}[t]{0.32\textwidth}
        \centering
        \includegraphics[height=1.1in]{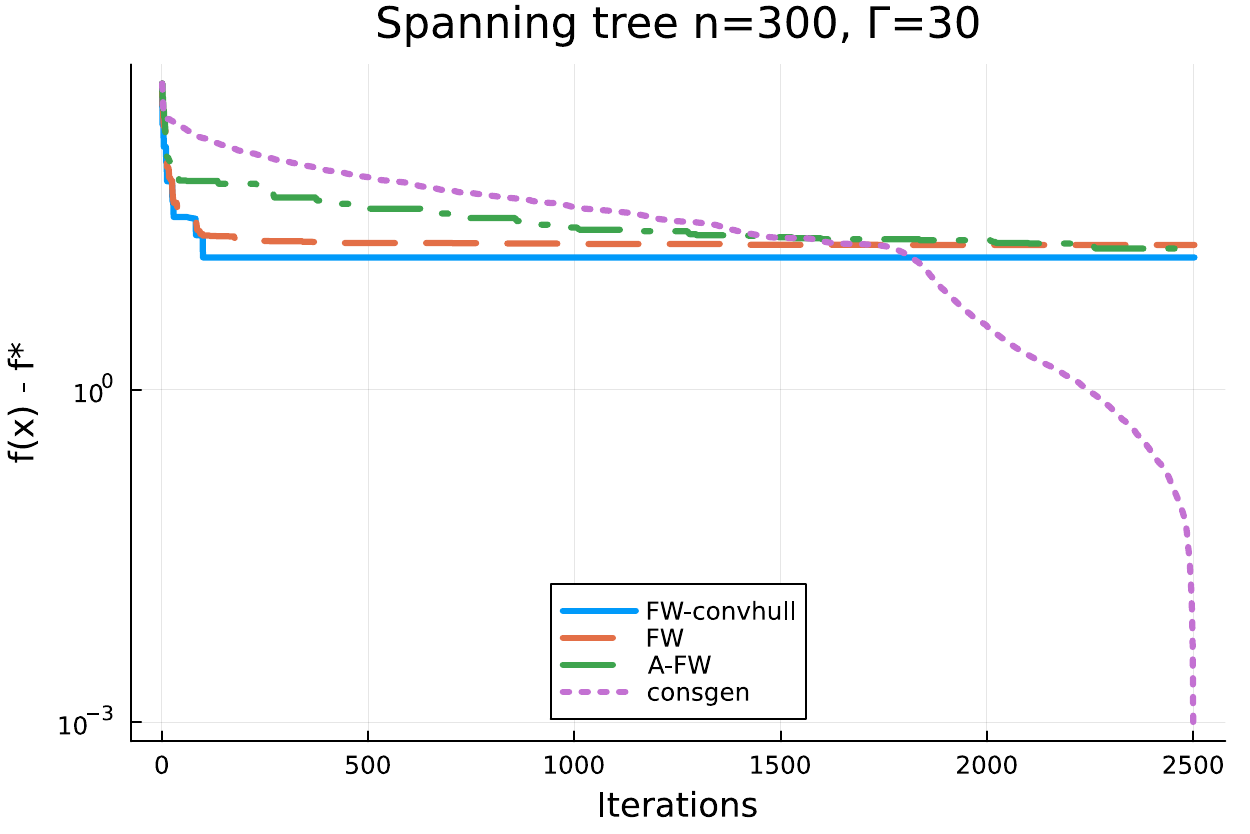}
    \end{subfigure}
    \begin{subfigure}[t]{0.32\textwidth}
        \centering
        \includegraphics[height=1.1in]{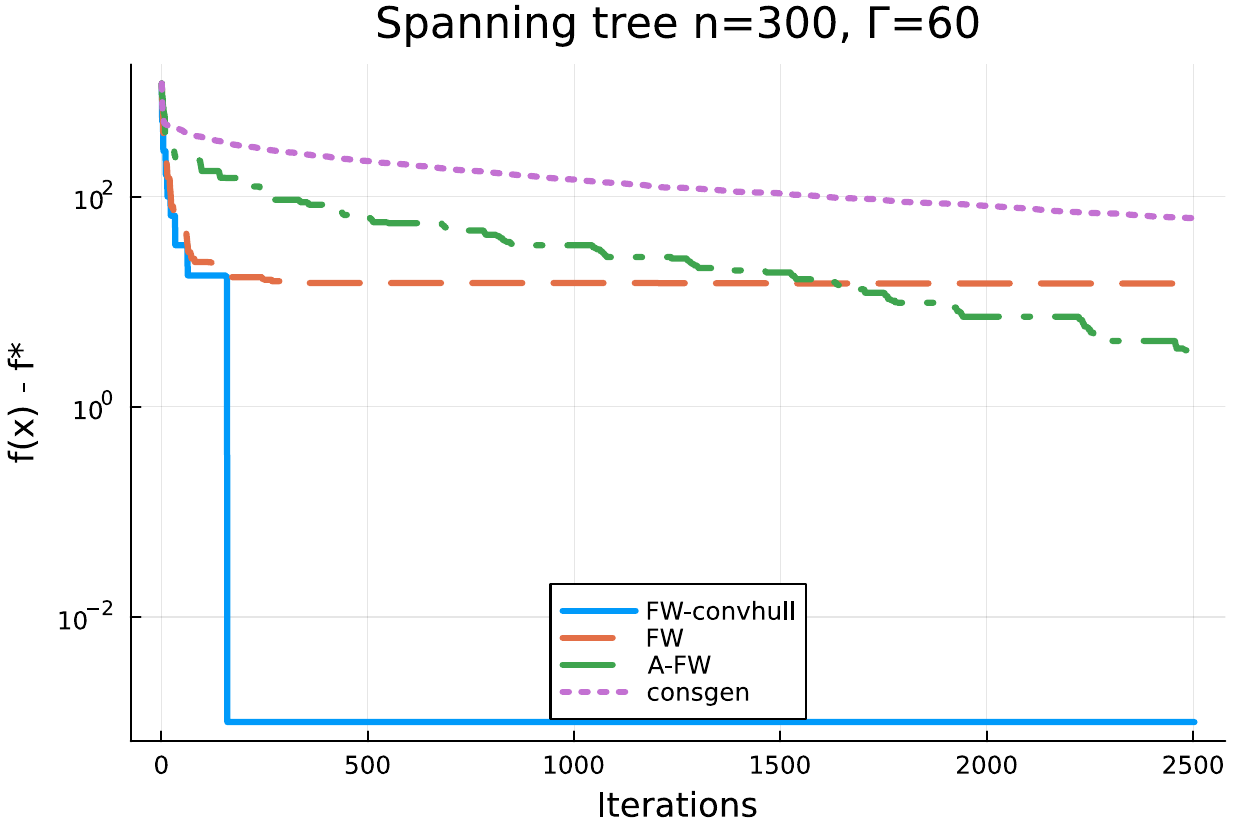}
    \end{subfigure}
    \begin{subfigure}[t]{0.32\textwidth}
        \centering
        \includegraphics[height=1.1in]{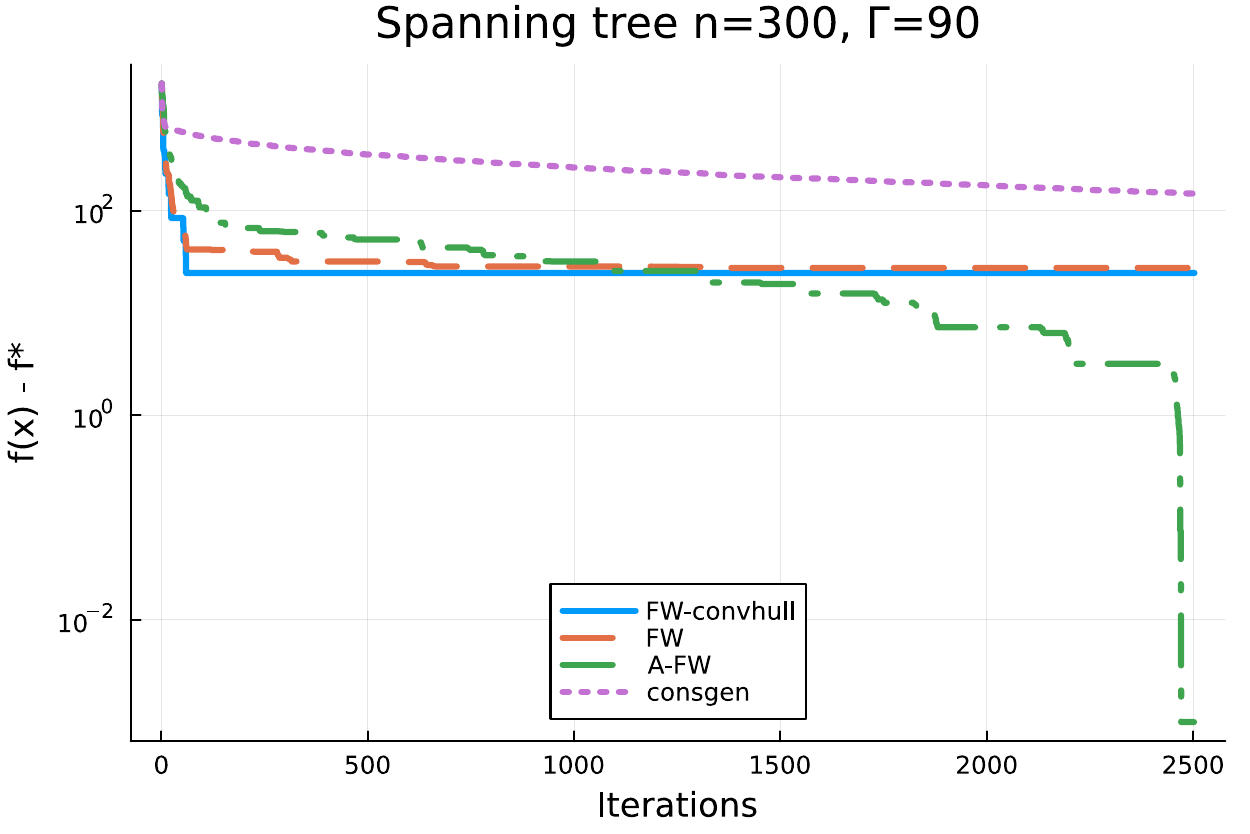}
    \end{subfigure}
    \caption{Primal value against iteration for $n=300$ for the ST problem for $\Gamma \in \{30,60,90\}$.}
    \label{fig:spanresultsgamma}
\end{figure*}

The relative performance of the different algorithms also carries over to use cases with more expensive LMOs. We assess them on robust TSP instances, where the LMO builds a MIP with lazy subtour elimination constraints with JuMP \citep{lubin2023jump}, then solved with GLPK \citep{makhorin2008glpk}. The results are presented in \cref{fig:tspresults}.
\begin{figure*}[t!]
    \centering
    \begin{subfigure}[t]{0.32\textwidth}
        \centering
        \includegraphics[height=1.1in]{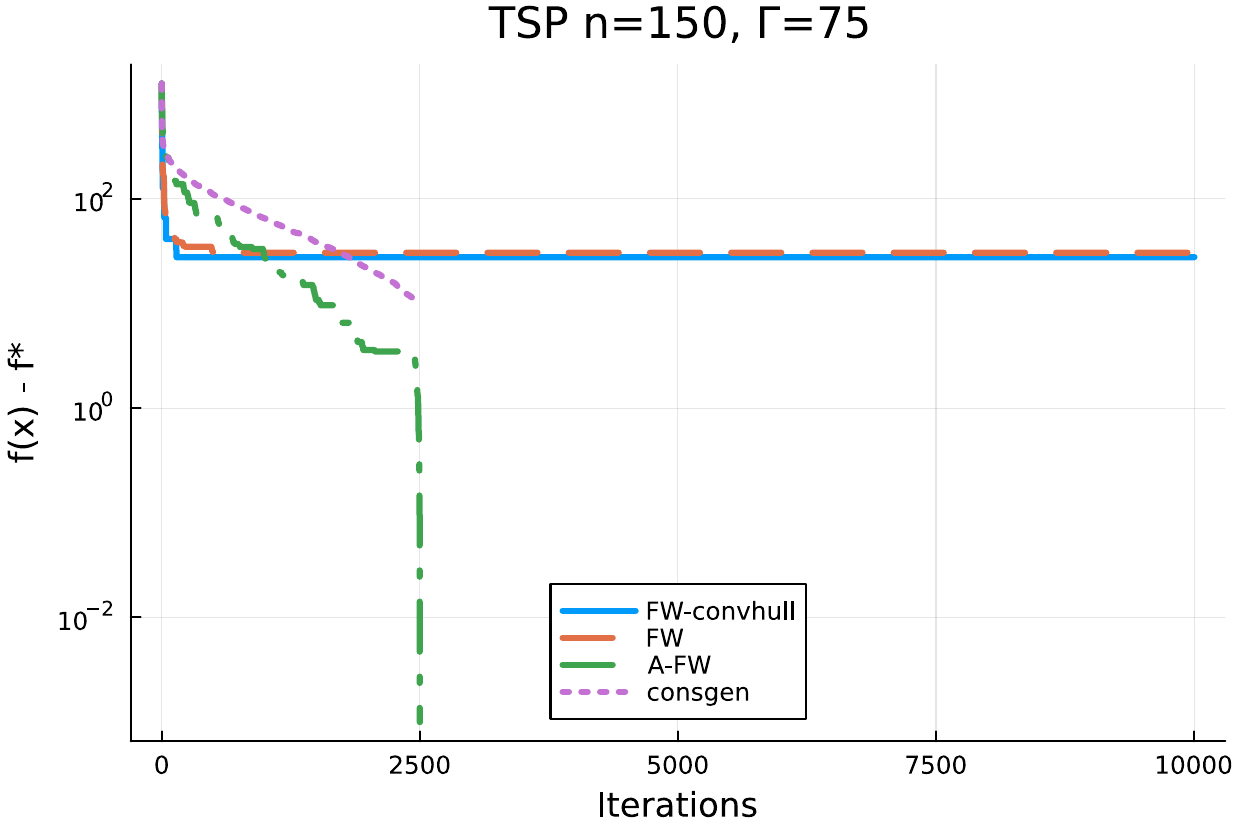}
    \end{subfigure}
    \begin{subfigure}[t]{0.32\textwidth}
        \centering
        \includegraphics[height=1.1in]{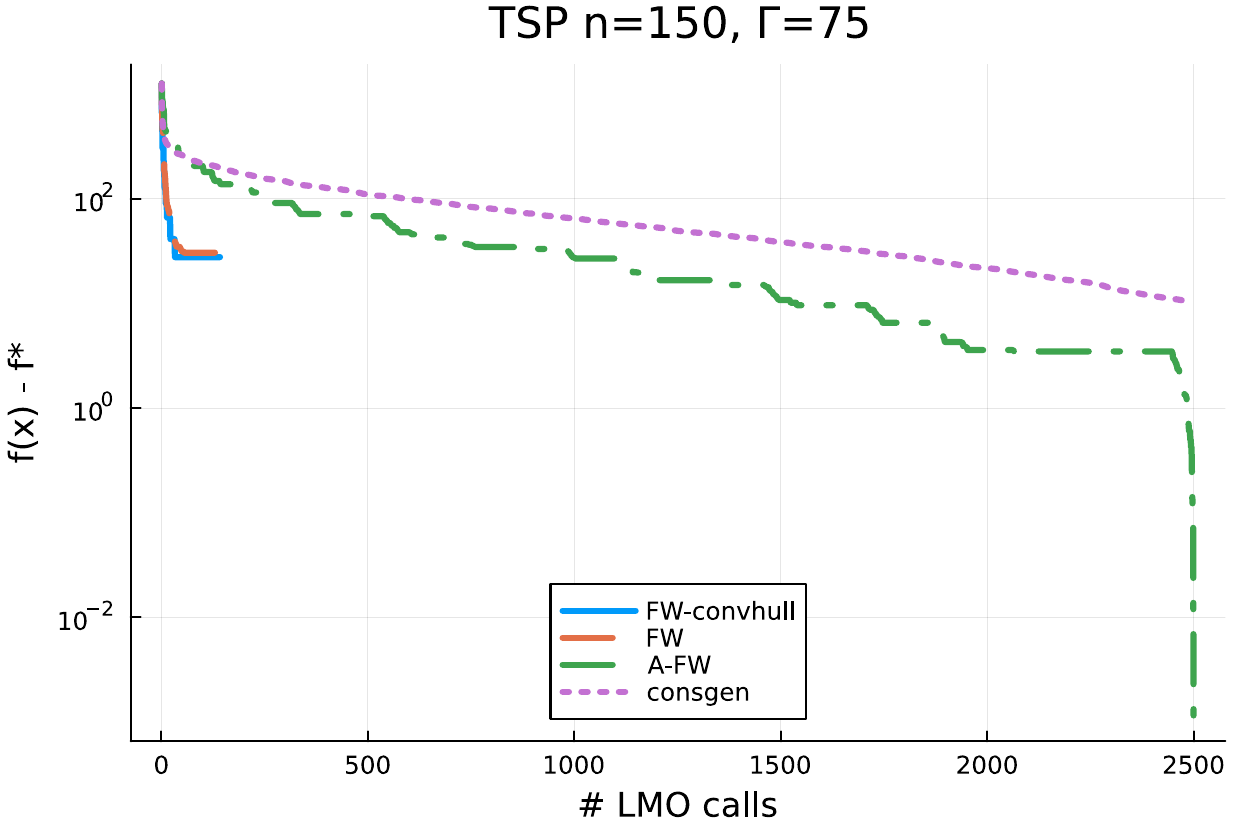}
    \end{subfigure}
    \begin{subfigure}[t]{0.32\textwidth}
        \centering
        \includegraphics[height=1.1in]{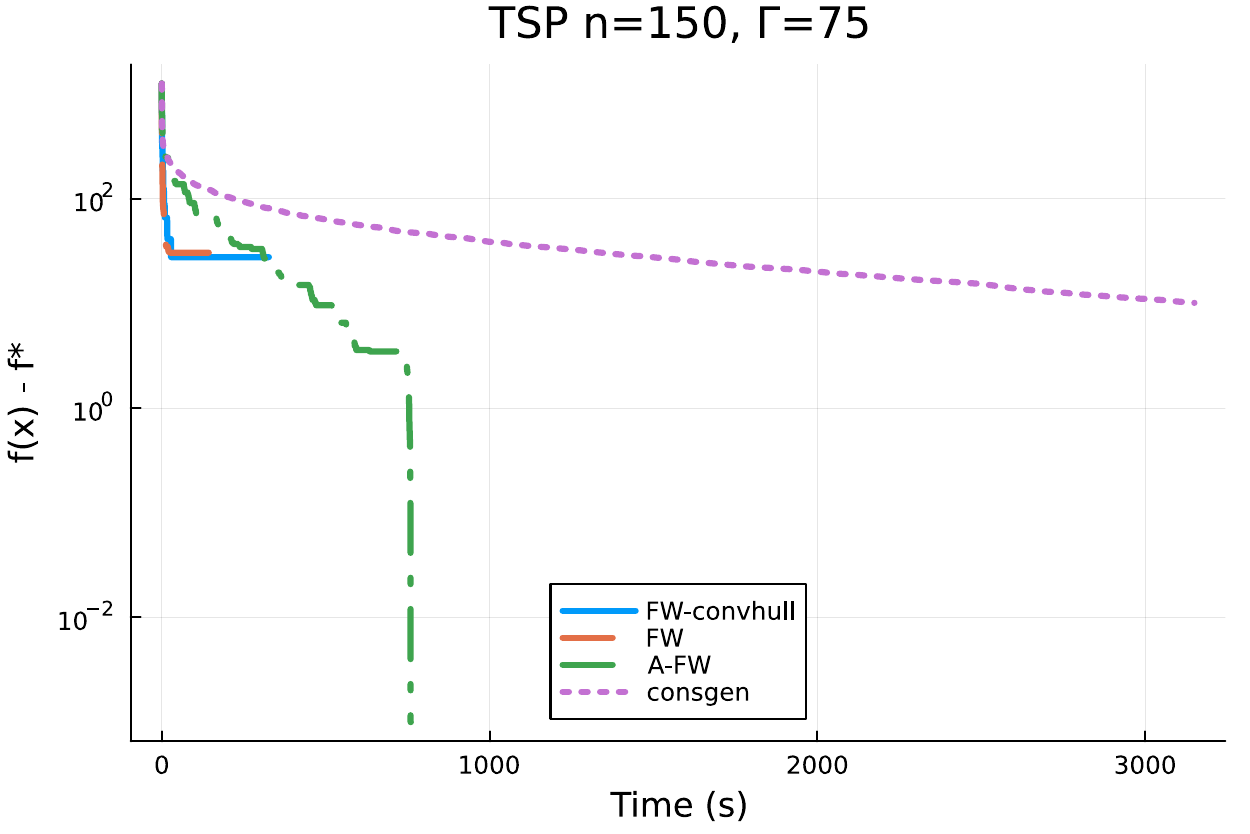}
    \end{subfigure}
    \caption{Comparison of the constraint generation and FW-based algorithms on a TSP example.}
    \label{fig:tspresults}
\end{figure*}
The results show in particular that constraint generation takes significantly more time than other methods, and in particular than adaptive FW, despite performing the same number of LMO calls that are notably more expensive than for spanning tree instances. Despite this important cost of LMO calls, constraint generation remains significantly slower due to the cost of the auxilary LPs.

\section{Conclusion}
In this work we solve the classical robust min-max problem with convex feasible region by a FW-type method applied to a smoothed version of the problem. The algorithm only uses a linear minimization oracle for the feasible region and does not require a compact description of it. By combining several concepts from the FW literature we could design an algorithm which outperforms the state-of-the-art on high-dimensional combinatorial problems under large uncertainty budget. Furthermore, the convergence analysis of the FW method leads to a theoretical bound on the number of oracle calls which are needed to achieve convergence.

While several oracle-based algorithms exist for classical robust optimization problems the literature for two-stage robust problems (2RO) is very sparse. While our framework can be incorporated into branch \& bound procedures to solve 2RO (see \citep{kammerling2020oracle}) it would be interesting to develop more direct oracle-based algorithms. Especially, the constraint-uncertainty case is not well-studied in this regard.




\bibliography{references}

\appendix

\end{document}